\documentclass[twoside,english]{amsart}
\usepackage[latin9]{inputenc}
\usepackage{amsthm}
\usepackage{setspace}
\usepackage{amssymb}
\numberwithin{equation}{section} 
\numberwithin{figure}{section} 
\theoremstyle{plain}
\theoremstyle{plain}
\newtheorem{thm}{Theorem}
  \theoremstyle{definition}
  \newtheorem{defn}[thm]{Definition}
  \theoremstyle{remark}
  \newtheorem*{acknowledgement*}{Acknowledgement}
  \theoremstyle{plain}
  \newtheorem{prop}[thm]{Proposition}
  \theoremstyle{plain}
  \newtheorem*{thm*}{Theorem}
  \theoremstyle{plain}
  \newtheorem{lem}[thm]{Lemma}
  \theoremstyle{plain}
  \newtheorem{cor}[thm]{Corollary}
  \theoremstyle{remark}
  \newtheorem*{rem*}{Remark}

\DeclareMathOperator{\rad}{rad}
\DeclareMathOperator{\minrad}{rmin}
\DeclareMathOperator{\maxrad}{rmax}

\DeclareMathOperator{\cdim}{cdim}

\begin{document}

\title{A ratio ergodic theorem for multiparameter non-singular actions}

\author{Michael Hochman}

\date{\today{}}
\begin{abstract}
We prove a ratio ergodic theorem for non-singular free $\mathbb{Z}^{d}$
and $\mathbb{R}^{d}$ actions, along balls in an arbitrary norm. Using
a Chacon-Ornstein type lemma the proof is reduced to a statement about
the amount of mass of a probability measure that can concentrate on
(thickened) boundaries of balls in $\mathbb{R}^{d}$. The proof relies
on geometric properties of norms, including the Besicovitch covering
lemma and the fact that boundaries of balls have lower dimension than
the ambient space. We also show that for general group actions, the
Besicovitch covering property not only implies the maximal inequality,
but is equivalent to it, implying that further generalization may
require new methods.
\end{abstract}

\curraddr{Fine Hall, Washington Road, Princeton University, Princeton, NJ 08544}

\email{hochman@math.princeton.edu}

\maketitle

\section{\label{sec:Introduction}Introduction}

Consider a non-singular action of a group $G$ on a standard $\sigma$-finite
measure space $(\Omega,\mathcal{B},\mu)$, which we denote $\omega\mapsto T^{g}\omega$;
we shall assume that the action is free and ergodic. From the action
on $\Omega$ there is induced an isometric linear action on $L^{\infty}$,
also denoted $T^{g}$, given by $T^{g}f(\omega)=f(T^{g^{-1}}\omega)$;
and this in turn induces an isometric linear action on the Banach
dual of $L^{\infty}$, whose restriction to $L^{1}$ is given by $\widehat{T}^{g}f=(T^{g^{-1}}f)\cdot\frac{d(g\mu)}{d\mu}$
(In the measure preserving case the Radon-Nikodym derivative is identically
$1$ and $\widehat{T}$ reduces to the usual Koopman operator, $\widehat{T}^{g}f(\omega)=f(T^{g^{-1}}\omega)$).

\subsection{The ratio ergodic theorem}

For $\mathbb{Z}$-actions, there is in this setting an analogue to
Birkhoff's ergodic theorem which is due to Hopf \cite{H37}, later
generalized by Hurewicz \cite{H44} to a {}``measureless'' statement
(see also \cite{Halmos46,Oxtoby48}), and by Chacon-Ornstein to the
operator setting \cite{CA60}. Hopf's ratio ergodic theorem states
that, for an ergodic $\mathbb{Z}$-action generated by a transformation
$T:\Omega\rightarrow\Omega$, for any $f,g\in L^{1}$ with $\int gd\mu\neq0$,
the following ratios converge almost surely: \[
R_{n}(f,g)=\frac{\sum_{k=0}^{n}\widehat{T}^{k}f}{\sum_{k=0}^{n}\widehat{T}^{k}g}\]
If in addition $T$ is conservative (i.e. has no nontrivial wandering
sets), or if the one-sided sum is replaced with the symmetric sum
from $-n$ to $n$, then the limit is the constant function $\int f/\int g$.
Note that for probability-preserving actions this is equivalent to
the usual ergodic theorem; $R_{n}$ becomes an ergodic average by
setting $g\equiv1$. For general actions this equivalence is false;
for example, for measure-preserving actions of an infinite measure
the ergodic averages converge to $0$, and not to the mean.

While the ergodic theorem for measure-preserving actions on probability
spaces has been broadly generalized to the group setting \cite{OW87},
the ratio theorem has not seen similar extensions, even to $\mathbb{Z}^{d}$-actions.
For a time it was thought no such extension was possible. The natural
thing to try in $\mathbb{Z}^{d}$ is to sum over the cubes $Q_{n}=[0;n]^{d}$,
but there is a counter-example, due to Brunel and Krengel, showing
that these ratios may diverge for $d>1$ \cite{K85}. However, recently
J. Feldman \cite{F07} proved a partial result for $\mathbb{Z}^{d}$,
showing that if the generators of the action act conservatively then
the ratio theorem holds for sums over the symmetric cubes $[-n;n]^{d}$.
This conservativity requirement is essential to the argument, and
is more restrictive than one would like, since there are certainly
actions that are conservative but whose generators are not (consider
for example the $\mathbb{Z}^{2}$ action generated by translation
by $\sqrt{2}$ and $\sqrt{3}$ on $\mathbb{R}$. The action is conservative
but each element of the action is a nontrivial translation, so no
cyclic subgroup acts conservatively).

Our main result is an unconditional ratio theorem for multiparameter
actions:
\begin{thm}
\label{thm:main}Let $\{T^{u}\}_{u\in\mathbb{Z}^{d}}$ be a free,
non-singular ergodic action on a standard $\sigma$-finite measure
space. Let $\left\Vert \cdot\right\Vert $ be a norm on $\mathbb{R}^{n}$
and let $B_{n}=\{u\in\mathbb{Z}^{d}\,:\,\left\Vert u\right\Vert \leq n\}$.
Then for every $f,g\in L^{1}$ with $\int g\neq0$, we have \[
R_{n}(f,g)=\frac{\sum_{u\in B_{n}}\widehat{T}^{u}f}{\sum_{u\in B_{n}}\widehat{T}^{u}g}\xrightarrow[n\rightarrow\infty]{}\frac{\int f}{\int g}\]
almost everywhere.
\end{thm}
A similar result holds for $\mathbb{R}^{d}$-actions.

\subsection{The Chacon-Ornstein lemma and amenable measures on $\mathbb{R}^{d}$}

The method of proof follows a two-step argument that is by now standard
and goes back to Hopf. With $g$ fixed, one first proves that $R_{n}(f,g)$
converges for $f$ in some dense subset $\mathcal{F}\subseteq L^{1}$.
Then one applies a maximal inequality to go from $\mathcal{F}$ to
its closure (we shall discuss maximal inequalities in more detail
below). In Feldman's proof the conservativity assumption is used to
construct a special family functions which is dense and for which
the ratios converge. We shall instead work with the larger subspace
generated by $g$ and bounded co-boundaries: \[
\mathcal{F}=\mbox{span}\{g,h-\widehat{T}^{u}h\,:\, u\in\mathbb{Z}^{d}\,,\, h\in L^{1}\cap L^{\infty}\}\]
A standard argument shows that $\mathcal{F}$ is dense in $L^{1}$
(see e.g. \cite{F07,A97}). Since $R_{n}(g,g)\equiv1$, convergence
of $R_{n}(f,g)$ for all $f\in\mathcal{F}$ will follow once it is
established for co-boundaries $f=h-\widehat{T}^{u}h$. For such $f$
some cancellation occurs in the sum $\sum_{v\in B_{n}}\widehat{T}^{v}f$,
and some algebra (given in section \ref{sec:Proof-of-the-ratio-theorem},
or see \cite{B83}) reduces the problem to the following variant of
the Chacon-Ornstein lemma, into the proof of which goes most of the
hard work:
\begin{thm}
\label{thm:chacon-ornstein}Under the hypotheses of theorem \ref{thm:main},
for any $h\in L^{\infty}\cap L^{1}$, and for any $t>0$,\[
\frac{\sum_{u\in B_{n+t}\setminus B_{n-t}}\widehat{T}^{u}h}{\sum_{u\in B_{n}}\widehat{T}^{u}h}\rightarrow0\]
almost surely.
\end{thm}
Note that for $d=1$ the numerator contains only two terms and it
suffices to show that the denominator tends to $\infty$; this follows
easily for conservative actions, while the non-conservative case can
be proved directly. On the other hand, for $d>1$ the number of terms
in the numerator is on the order of $n^{d-1}$, and when the measure
is infinite the denominator satisfies $\frac{1}{n^{d}}\sum_{u\in B_{n}}\widehat{T}^{u}h\rightarrow0$.
Thus a more sophisticated argument is necessary. 

Our proof of theorem \ref{thm:chacon-ornstein} applies the transference
principle to reduce theorem \ref{thm:chacon-ornstein} to a geometric
statement about the amount of mass which can concentrate on boundaries
of balls for finite measures in $\mathbb{R}^{d}$. In the proof we
use two facts related to the finite dimension of $\mathbb{R}^{d}$
(with combinatorial analogs in $\mathbb{Z}^{d}$). One is the Besicovitch
covering lemma, about which we shall have more to say below. The other
is (a variant of) the fact that the boundary of balls in $\mathbb{R}^{d}$
are manifolds of lower dimension, which is closely related to finite
topological dimension of $\mathbb{R}^{d}$. This property has apparently
not been exploited before in this context. It is also worth noting
that our method does not require us to distinguish between the conservative
and non-conservative case.

The same methods used in the proof of theorem \ref{thm:chacon-ornstein}
give the following theorem, which can be placed in geometric measure
theory and may be of independent interest. 
\begin{thm}
Let $\nu$ be a Borel probability measure on $\mathbb{R}^{d}$ and
let $B_{r}(x)$ be the ball of radius $r$ around $x\in\mathbb{R}^{d}$
with respect to some fixed norm. Then\[
\lim_{r\searrow0}\frac{\nu(\partial B_{r}(x))}{\nu(B_{r}(x))}=0\]
for $\nu$-almost every $x$.
\end{thm}

\subsection{The maximal inequality}

As we have mentioned already, in order to derive the ratio theorem
from theorem \ref{thm:chacon-ornstein} one uses the maximal inequality,
which is the second subject of this paper. We shall denote by $B_{n}$
an increasing sequence of finite subsets of $G$ which satisfy $1_{G}\in B_{n}$
for all $n$.
\begin{defn}
\label{def:maximal-inequality}An ergodic action of $G$ admits a
ratio maximal inequality (with respect to $(B_{n})$) if, for every
$0\leq g\in L^{1}$ there is a constant $M$ such that, for every
$0\leq f\in L^{1}$ and $\varepsilon>0$, \[
\mu_{g}\{\omega\in\Omega\,:\,\sup_{n}R_{n}(f,g)>\varepsilon\}\leq\frac{M}{\varepsilon}\int fd\mu\]
where $d\mu_{g}=g\cdot d\mu$. We say that there is a maximal inequality
for $G$ (with respect to $(B_{n})$) if every action admits a maximal
inequality.
\end{defn}
Notice that when $\mu(\Omega)=1$, we can take $g\equiv1$. Then $\mu_{g}=\mu$
and $R_{n}(f,g)$ are the ergodic averages of $f$, so the ratio maximal
inequality reduces to the usual maximal inequality. Note also that
we allow the constant $M$ to depend on $g$, since this is what is
used in the proof of the ergodic theorem.

The ordinary maximal inequality for probability-preserving actions
of amenable groups is known to hold quite generally \cite{OW87},
but this is not so in the non-singular case. Indeed, if the Krengel-Brunel
counter-example is examined closely it is evident that there is a
dense class of functions $f$ for which the ratio theorem holds. The
problem must be that the maximal inequality fails. This is closely
related to the fact that the sum is over one-sided cubes $Q_{n}=[1;n]^{d}$,
which fail to satisfy the Besicovitch covering property:
\begin{defn}
\label{def:besicovitch}A sequence $(B_{n})$ of subsets of $G$ satisfies
the Besicovitch covering property with constant $C$ if the following
holds. If $E\subseteq G$ is finite, and for each $g\in E$ we are
given a translate $B_{n(g)}g$ of one of the $B_{n}$'s, then there
is a subfamily of these translates which covers $E$ and such that
no point in $G$ is covered more than $C$ times.
\end{defn}
This geometric property has found many applications in analysis; an
excellent source on this is \cite{G75}. That it implies the ratio
maximal inequality was first shown by M. Becker \cite{B83} for balls
$B_{n}\subseteq\mathbb{R}^{d}$ in a given norm. A maximal inequality
relying on the Besicovitch property was later also established by
E. Lindenstrauss and D. Rudolph for a more general class non-singular
group actions \cite{LR06}. A short proof of the general case can
be found in Feldman's paper \cite{F07}. Other applications of the
Besicovitch property to ergodic theory appear in \cite{H06}.

It is thus known that the Besicovitch property implies the ratio maximal
inequality. It has apparently not been observed before that it is
also necessary.
\begin{thm}
\label{thm:maximal-characterization}Let $G$ be a countable group
and $B_{r}\subseteq G$ an increasing sequence of symmetric sets with
$\cap B_{r}=\{e\}$. Then there is a ratio maximal inequality for
$G$ if and only if $B_{n}$ satisfies the Besicovitch property.
\end{thm}
Actually, more is true: if $B_{n}$ is not Besicovitch then the ratio
maximal inequality fails for every free action of the group. Contrast
this with the usual maximal inequality, which holds for any measure-preserving
action of an amenable group on a probability space, as long as the
averages are taken over a tempered F\o{}lner sequence \cite{L01,W03}. 

One should note that the Besicovitch property is rather rare. It fails,
for example, for the Heisenberg group when $B_{n}$ are balls with
respect to several natural metrics \cite{R04}. 

It is not clear what all this says about the ratio ergodic theorem.
For probability-preserving actions of amenable groups the ratio theorem
along tempered F\o{}lner sequences follows from the ordinary ergodic
theorem. At the same time, the ratio maximal inequality fails, as
we saw before. The ratio ergodic theorem does hold, for trivial reasons,
for dissipative actions (e.g. on atomic measure spaces). This leaves
the hope that a ratio theorem may persist in a more general setting.

The rest of this paper is organized as follows. In the next section
we discuss the Besicovitch property and prove theorem \ref{thm:maximal-characterization}.
In section \ref{sec:The-doubling-property} we discuss some covering
and disjointification lemmas. In section \ref{sec:Non-Concentration-of-mass}
we define coarse dimension and prove our main tool about concentration
of measures on ball boundaries. In section \ref{sec:Proof-of-the-ratio-theorem}
we complete the proofs of theorems \ref{thm:chacon-ornstein} and
the ratio theorem, \ref{thm:main}.
\begin{acknowledgement*}
I would like to thank Benjamin Weiss and Elon Lindenstrauss for introducing
me to this problem and for their useful comments. Thanks also to J.
Bourgain for his comments.
\end{acknowledgement*}

\section{\label{sec:Besicovitch}The Besicovitch lemma and the Maximal inequality}

In this section we prove theorem \ref{thm:maximal-characterization}.
We shall reformulate the Besicovitch covering property for metric
spaces and present it in several equivalent forms. An excellent source
on these matters is \cite{G75}.

Given a metric space $(X,d)$ we denote by $B_{r}(x)$ the open ball
of radius $r$ centered at $x$. We think of balls as carrying with
them the information about their center and radius, which are not
in general determined by the ball as a set. 

A finite family of balls $\mathcal{U}=\{B_{r(i)}(x_{i})\,:\,1\leq i\leq N\}$
is called a \emph{carpet} over $\{x_{1},\ldots,x_{N}\}$. It is sometimes
convenient to regard carpets as ordered sets. Note that the statement
that $\mathcal{U}$ is a carpet over $E$ is stronger than the statement
that it covers $E$, since the former asserts that each $x\in E$
is the center of some ball in $\mathcal{U}$, whereas the latter only
says that $x$ belongs to some ball.

We say that a collection of sets has \emph{multiplicity} $\leq m$
if every point is contained in at most $m$ elements of the collection. 

A metric space satisfies the \emph{Besicovitch property} with constant
$C$ \cite{B45,G75} if for any carpet over $E$ there exists a sub-carpet
which covers $E$ and has multiplicity $\leq C$. The main example
for this is $\mathbb{R}^{d}$ with a norm-induced metric; this was
shown by Morse \cite{M47}. A more accessible proof can be found in
\cite{G75} or can be deduced from proposition \ref{pro:Rd-lemma}
below.

This definition of the Besicovitch property is consistent with the
one in the introduction if $X=G$ is a group, $d$ is a right-invariant
metric on $G$ and $B_{n}$ are the balls of radius $n$ around the
group's identity element. We shall allow ourselves to switch freely
between these two formalisms, which are notationally identical.

We say that a sequence $B_{r(i)}(x_{i})$ is \emph{incremental }if
$r(i)$ is decreasing and $x_{i}\notin\cup_{j<i}B_{r(j)}(x_{i})$. 

The Besicovitch property has several equivalent forms which are useful
in applications.
\begin{prop}
\label{pro:Besicovitch-characterizations}Let $X$ be a metric space
and $C\in\mathbb{N}$. The following are equivalent:
\begin{enumerate}
\item \label{enu:Besicovitch}$X$ has the Besicovitch property with constant
$C$.
\item \label{enu:Bes-low-freq}For any carpet $\mathcal{U}$ over $E$ and
$A,B\subseteq E$, if $t>0$ and $|A\cap F|/|B\cap F|<t$ for $F\in\mathcal{U}$,
then $|A|/|B|<Ct$.
\item \label{enu:Bes-high-freq}For any carpet $\mathcal{U}$ over $E$
and $A,B\subseteq E$, if $t>0$ and $|A\cap F|/|B\cap F|>t$, for
$F\in\mathcal{U}$, then $|A|/|B|>\frac{1}{C}t$.
\item \label{enu:Bes-incremental}Every incremental sequence has multiplicity
$\leq C$.
\item \label{enu:Bes-incremental-cover}For any carpet $\mathcal{U}$ over
$E$ there is an incremental sequence of sets from $\mathcal{U}$
covering $E$ and with multiplicity $\leq C$
\end{enumerate}
\end{prop}
\begin{proof}
(\ref{enu:Besicovitch}) implies (\ref{enu:Bes-low-freq}): Using
(\ref{enu:Besicovitch}) we may pass to a sub-collection $\{F_{i}\}_{i\in I}\subseteq\mathcal{U}$
with multiplicity $\leq C$, and which covers $E$, and hence covers
$A$ and $B$. It now follows that \[
|A|\leq\sum_{i}|A\cap F_{i}|<t\sum|B\cap F_{i}|\leq Ct|\cup_{i}(B\cap F_{i})|=Ct|B|\]

(\ref{enu:Bes-low-freq}) and (\ref{enu:Bes-high-freq}) are equivalent
on reversing the roles of $A$ and $B$.

(\ref{enu:Bes-low-freq}) implies (\ref{enu:Bes-incremental}): Let
$B_{r(1)}(x_{1}),\ldots,B_{r(N)}(x_{N})$ be an incremental sequence.
For each $i$, note that if $j<i$ then $x_{i}\notin B_{r(j)}(x_{j})$
because the sequence is incremental, and therefore $x_{i}\notin B_{r(j)}(x_{j})$.
Suppose $y\in\cap B_{r(i)}(x_{i})$. Choose $\varepsilon>0$ so that
$x_{i}\notin B_{\varepsilon}(y)$; setting $A=\{x_{1},\ldots,x_{N}\}$
and $B=\{y\}$ we see that for any $t>1$ the hypothesis of (2) is
satisfied with respect to the carpet $\{B_{\varepsilon}(y)\}\cup\{B_{r(i)}(x_{i})\,:\,1\leq i\leq N\}$
over $\{x_{1},\ldots,x_{N},y\}$, implying that $N=|A|/|B|<Ct$ for
every $t>1$, and the claim follows.

(\ref{enu:Bes-incremental}) implies (\ref{enu:Bes-incremental-cover}):
Let $\mathcal{U=\{}B_{r(1)}(x_{1}),B_{r(2)}(x_{2}),\ldots,B_{r(N)}(x_{N})\}$
be a carpet over $E=\{x_{1},\ldots,x_{N}\}$. Without loss of generality
we may assume that $r(1)\geq r(2)\geq\ldots\geq r(N)$. Iterate over
$i$ from $1$ to $N$ and at each stage select or discard the set
$B_{r(i)}(x_{i})$ according to whether $x_{i}$ belongs to the union
of the sets selected previously or not. We obtain an incremental sequence
which covers $E$, and by (\ref{enu:Bes-incremental}) has multiplicity
$\leq C$.

The implication (\ref{enu:Bes-incremental-cover})$\Rightarrow$(\ref{enu:Besicovitch})
is trivial.
\end{proof}
We can now prove theorem \ref{thm:maximal-characterization}:
\begin{thm*}
Let $G$ be a countable group and $B_{r}\subseteq G$ an increasing
sequence of symmetric sets with $\cap B_{r}=\{e\}$. Then $G$ satisfies
a maximal inequality for sums over $B_{r}$ if and only if $X$ satisfy
the Besicovitch property with respect to $B_{r}$.\end{thm*}
\begin{proof}
One direction is the maximal inequality of Becker and of Lindenstrauss
and Rudolph \cite{B83,LR06}.

Conversely, suppose the Besicovitch property fails. Given an action
of $G$ and function $f,h\in L^{1}$ let us write \[
C(f,h)=\frac{\mu_{h}\{\sup_{n}R_{n}(f,h)>\varepsilon\}}{\int fd\mu}\]
where $d\mu_{h}=h\cdot d\mu$. We are out to show that for some $h\in L^{1}$
this quantity is not bounded in $f$.

We start with the action of $G$ on itself by left translation, $T^{g}x=gx$,
and let $\mu$ be Haar (counting) measure, which is clearly preserved.
By proposition \ref{pro:Besicovitch-characterizations}, for every
$M>0$ there is a $t>0$ and finite sets $U,V\subseteq G$, and $n(g)\in\mathbb{N}$
for $g\in U\cup V$, so that \[
|U\cap B_{n(g)}g|/|V\cap B_{n(g)}g|>t\]
for $g\in U\cup V$, but $|U|/|V|<t/M$. Let $f=1_{U}$ and $h=1_{V}$.
Then \[
\{x\in V\,:\,\sup_{n}\frac{\#\{g\in B_{n}\,:\, T^{g}x\in U\}}{\#\{g\in B_{n}\,:\, T^{g}x\in V\}}>t\}=V\]
so \[
\mu_{h}(x\in G\,:\,\sup_{n}\frac{\sum_{g\in B_{n}}f(T^{g}x)}{\sum_{g\in B_{n}}h(T^{g}x)}>t)=\mu_{h}(V)=|V|>\frac{M}{t}|U|=\frac{M}{t}\int fd\mu\]
We have found that for each $M>0$ there are $f,h\in L^{2}(G)$ with\[
C(f,h)>M\]
This is already enough to conclude that the ratio maximal inequality
cannot hold with a constant which is independent of $h$.

We next want to show that for every action of $G$ on a measure space
$(\Omega,\mathcal{F},\mu)$, there is a fixed $h$ with $\sup_{f\in L^{1}}C(f,h)=\infty$.
We prove this for the case that the measure space is non-atomic and
the action measure-preserving. The proof for the atomic case is simpler,
so we omit it. 

We construct by induction functions which will establish our claim.
Suppose we have $0\leq f_{k},h\in L^{1}(\Omega)$ for $1\leq k\leq n$
such that $C(f_{k},h)>k$. Using what we know about the action of
$G$ on itself, we can find $0\leq f',h'\in L^{1}(G)$ so that $C(f',h')>n+1$.
Below we show how to merge $h,h'$ into a function $h''\in L^{1}(\Omega)$,
and to construct a function $f_{n+1}\in L^{1}(\Omega)$ derived from
$f'$, so that $C(f_{k},h'')>k$ for $1\leq k\leq n+1$ and $\left\Vert h-h''\right\Vert <\varepsilon$,
where $\varepsilon$ is a parameter which can be chosen arbitrarily
small. Once this is done, we can iterate the process and pass to a
limit function $h_{*}\in L^{1}(\Omega)$ which, for the sequence $f_{k}$
constructed, satisfies $C(f_{k},h_{*})\rightarrow\infty$, completing
the proof of the theorem. 

Fix $\varepsilon>0$. We may assume that the functions $f',h'$ that
we have found on $G$ are supported inside $B_{N_{0}}$ for some $N_{0}$
and that if $R_{i}(f',h')(g)>n+1$ for some $i$ and $g\in G$ then
$g\in B_{N_{0}}$. If $\omega\in\Omega$ let $i_{k}(\omega)$ be the
first index so that $R_{i_{k}(\omega)}(f_{k},h)(\omega)>k$. Since
$i_{k}(\cdot)$ is measurable for $k=1,\ldots,n$, there is some $N_{1}$
so that $\mu_{h}(i_{k}>N_{1})<\varepsilon$ for each $k$. Set $N=\max\{N_{0},N_{1}\}$.

Using the fact that the action is non-atomic and free, we can find
a set $A\subseteq\Omega$ with positive $\mu$-measure, so that $gA\cap g'A=\emptyset$
whenever $g,g'\in B_{N}$, and so that \[
\widetilde{A}=\cup_{g\in B_{N}}gA\]
has measure less than $\varepsilon$, both with respect to $\mu$
and with respect to $\mu_{h}$ \cite{W03}. By the choice of $N$,
if $h''$ is a function that differs from $h$ only on $\widetilde{A}$
then $C(f,h'')>M-\varepsilon$, because for $\omega\in\Omega\setminus\widetilde{A}$,
we have $\widehat{T}^{g}h(\omega)=\widehat{T}^{g}h''(\omega)$ as
long as $g\in B_{N}$, implying $R_{i_{k}(\omega)}(f,h')(\omega)=R_{i_{k}(\omega)}(f,h)(\omega)$
outside of $\widetilde{A}$. 

Define $h''(\omega)=h'(g)$ for $\omega\in gA$ and $g\in B_{N_{0}}$,
and $h''=h$ otherwise. By the above, $C(f_{k},h'')>C(f_{k},h)-\varepsilon$
for $k=1,\ldots,n$ and $\left\Vert h-h''\right\Vert =\int_{A}|h-h''|d\mu$,
which can be made $>k$ and $<2^{-n}$ respectively by choosing $\varepsilon$
small enough. 

Finally, define the function $f_{n+1}(\omega)=f'(g)$ for $\omega\in gA$
and $g\in B_{N_{0}}$, and $0$ otherwise. Since on $B_{N}A$ we have
$R_{i}(f_{n+1},h'')=R_{i}(f',h')$ for $i\leq N_{0}$, we have that
$C(f_{n+1},h'')>n+1$ (notice that $C(\cdot,\cdot)$ is invariant
under scaling of $\mu$, which explains why constructing $f'',h''$
on a part of the measure space which is small with respect to $\mu$
does not ruin the property $C(f',h')>n+1$).
\end{proof}

\section{\label{sec:The-doubling-property}The doubling property and disjointification}

Another property of metric spaces which is related (but not equivalent
to) the Besicovitch property is the doubling condition. Let $(X,d)$
be a metric space, and suppose we are given a measure on $X$ which
we denote by $|\cdot|$; in our setting it will be Haar measure, and
for $\mathbb{Z}^{d}$ will denote the usual counting measure. We say
that $(X,d)$ satisfies the \emph{doubling condition }with constant
$D$ if, for every ball $B_{r}(x)$ we have $|B_{2r}(x)|\leq D|B_{r}(x)|$.
This is satisfied for the groups $\mathbb{Z}^{d},\mathbb{R}^{d}$
for any norm; for finitely generated groups with word metric this
condition is equivalent to polynomial growth.

In this section we derive some covering lemmas based on the doubling
and Besicovitch properties. For this we require some more notation.
Write $\rad B$ for the radius of a ball $B$, and if $\mathcal{U}$
is a collection of balls we write $\maxrad\mathcal{U}$ and $\minrad\mathcal{U}$
for the maximal and minimal radii of balls in $\mathcal{U}$, respectively. 

We say that $\mathcal{U}$ is \emph{well-separated} if for every two
balls in $\mathcal{U}$ are at distance at least $\minrad\mathcal{U}$
from each other.

The doubling condition together with the Besicovitch property imply
the following standard covering result which can be found e.g. in
\cite{G75}.
\begin{lem}
\label{lem:disjointification}Let $X$ be a metric space with a measure,
and suppose it satisfies the Besicovitch property with constant $C$
and the doubling condition with constant $D$. Then for every finite
$E\subseteq X$ and every carpet $\mathcal{U}$ over $E$ there is
a sub-collection $\mathcal{V}\subseteq\mathcal{U}$ which covers $E$
and which can be partitioned into $\chi=CD^{2}+1$ sub-collections,
each of which is well-separated. \end{lem}
\begin{proof}
We begin with a few observations. Let $x\in X$ an let $\mathcal{W}$
be a collection of $n$ balls of radius $R$ centered inside $B_{3R}(x)$,
and suppose $\mathcal{W}$ has multiplicity $\leq C$. Then $\cup\mathcal{W}\subseteq B_{4R}(x)$,
so \[
n|B_{R}(x)|\leq C\cdot|B_{4R}(x)|\leq D^{2}\cdot C\cdot|B_{R}(x)|\]
hence $n\leq CD^{2}=\chi-1$. 

Next, Suppose $\mathcal{W}$ consists of balls of radius $\geq R$
which intersect $B_{2R}(x)$, and suppose $\mathcal{W}$ has multiplicity
$\leq C$. By replacing each ball $B\in\mathcal{W}$ with a ball of
radius $R$ contained in $B$ and centered within $B_{3R}(x)$, we
conclude again that $|\mathcal{W}|\leq\chi-1$. 

We now prove the lemma. By (\ref{enu:Bes-incremental-cover}) of proposition
\ref{pro:Besicovitch-characterizations}, choose an incremental sequence
$U_{1},\ldots,U_{n}\in\mathcal{U}$ covering $E$, and assign colors
$1,2,\ldots,\chi$ to the $U_{i}$ as follows. Color $U_{1}$ arbitrarily.
Assuming we have colored $U_{1},\ldots,U_{k}$ consider $U_{k+1}$.
By the above, $U_{k+1}$ cannot be within distance $\rad U_{k}$ of
more than $\chi-1$ of the balls we have already colored, so there
is a color which we can assign to it without violating the coloring
condition. When all the balls are colored, set $\mathcal{V}_{k}=$the
balls colored $k$. Clearly each collection is well-separated.
\end{proof}
We denote by $\chi(X)$ the smallest constant $\chi$ for which $X$
satisfies the conclusion of the proposition. Clearly, if $Y\subseteq X$
then $\chi(Y)\leq\chi(X)$. If $X$ satisfies the hypotheses of the
proposition then $\chi(X)\leq CD^{2}+1$, so this bound holds for
any $Y\subseteq X$, even though $Y$ may no longer satisfy the doubling
condition.
\begin{cor}
\label{cor:measure-disjointification}In the notation of the previous
lemma, assume there is given a finite measure $\mu$ supported in
a set $E$. Then there is a well-separated subset of $\mathcal{U}$
which covers a set of mass $\geq\frac{1}{\chi}\mu(E)$.\end{cor}
\begin{proof}
Color the balls as in the previous lemma. Each monochromatic collection
of balls is well-separated and since there are $\chi$ colors, and
the union covers $E$, one color class covers a $1/\chi$-fraction
of the mass.
\end{proof}
Our next objective is a lemma like the above except that, instead
of capturing mass in a well-separated collection of balls, we do so
with spheres, or more precisely thick spheres. This is too much to
hope for in general, but we can do it under the hypothesis that the
balls we begin with contain some fraction of their mass on their boundaries. 

We For a metric space $X$, the $t$-boundary of a ball $R_{r}(x)$
is defined for $t\leq r$ by \[
\partial_{t}B_{r}(x)=B_{r+t}(x)\setminus B_{r-t}(x)\]
this is called a \emph{thick sphere}; we say that its radius is $r$
and thickness $t$, and agree that the sphere carries this information
with it ($r,t$ are not determined from $\partial_{t}B_{r}(x)$ in
general). Also write $\partial B_{r}(x)$ for the usual topological
boundary of $B_{r}(x)$, which is a \emph{sphere}. We apply these
operations to collections element-wise, i.e. if $\mathcal{U}$ is
a collection of balls we write $\partial\mathcal{U}=\{\partial B\,:\, B\in\partial U\}$,
etc.

If $\mathcal{U}$ is a collection of spheres we define $\minrad\mathcal{U}$,
$\maxrad\mathcal{U}$ in the same way as for balls. For $R>0$, we
say the collection is \emph{$R$-separated} if every two members are
at distance at least $R$ from each other. If this is true for $R=\minrad\mathcal{U}$
we say the collection is \emph{well-separated}. Thus the $t$-boundaries
of an $R$-separated collection of balls is $(R-2t)$-separated. Note
that an $R$-separated family of spheres may be nested: although the
spheres are disjoint, the corresponding balls may be contained in
each other. 

A sequence $\mathcal{U}_{1},\ldots,\mathcal{U}_{p}$ of carpets over
$E$ is called a \emph{stack}, and $p$ is its height. 

Given a measure $\mu$, a set $F$ and a collection of sets $\mathcal{U}$,
we say that $\mathcal{U}$ covers an $\varepsilon$-fraction of $F$
if $\mu(F\cap(\cup\mathcal{U}))\geq\varepsilon\mu(F)$. 
\begin{lem}
\label{lem:sphere disjointification}Let $X$ be a metric space satisfying
the Besicovitch and doubling properties and let $\chi=\chi(X)$. For
$0<\varepsilon,\delta<1$, $d\in\mathbb{N}$ let $p=\left\lceil \frac{2\chi}{\varepsilon\delta}\right\rceil $,
and suppose that
\begin{enumerate}
\item $\mu$ is a finite measure on $X$.
\item $F\subseteq X$ is finite and $\mu(F)>\delta\mu(X)$.
\item $\mathcal{U}_{1},\mathcal{U}_{2},\ldots,\mathcal{U}_{p}$ is a stack
over $F$ with $\minrad\mathcal{U}_{i}\geq\maxrad\mathcal{U}_{i-1}$.
\item $\mu(\partial_{1}B)>\varepsilon\mu(B)$ for each $B\in\cup_{i}\mathcal{U}_{i}$.
\end{enumerate}
Then there is an integer $k\geq1$ and a sub-collection $\mathcal{V}\subseteq\cup_{i\geq k}\mathcal{U}_{i}$
of spheres so that:, \renewcommand{\theenumi}{\alph{enumi}}
\begin{enumerate}
\item $\partial\mathcal{V}$ is well-separated, 
\item For $r=\maxrad\mathcal{U}_{k-1}$, the set $\cup_{B\in\mathcal{V}}\partial_{2r}B$
contains more than $1/2$ of (the $\mu$-mass of) $F$.\renewcommand{\theenumi}{\arabic{enumi}}
\end{enumerate}
\end{lem}
\begin{rem*}
If we assume that $\minrad\mathcal{U}_{k}>4\maxrad\mathcal{U}_{k-1}$
then we can conclude that the collection $\partial_{2r}B$, $B\in\mathcal{V}$
is pairwise disjoint.\end{rem*}
\begin{proof}
The proof follows the usual Vitali-like exhaustion scheme. We describe
a recursive procedure for constructing $\mathcal{V}$, and show that
it will eventually terminate with a suitable collection. Our induction
hypothesis is that at the $k$-th stage we have constructed a collection
$\mathcal{V}\subseteq\cup_{i>p-k}\mathcal{U}_{i}$ with $\partial\mathcal{V}$
well-separated, and with $\mu(\cup_{B\in\mathcal{V}}\partial_{2r}B)\geq\frac{\varepsilon\delta}{2\chi}\cdot k\cdot\mu(X)$
for $r=\maxrad\mathcal{U}_{k-1}$. 

We begin for $k=0$ with $\mathcal{V}=\emptyset$, which satisfy this
trivially. Assuming we have completed the $k$-th stage, let $r=\maxrad\mathcal{U}_{k-1}$.
Distinguish two cases.

If $\mu(F\cap\cup_{B\in\mathcal{V}}\partial_{2r}B)>\frac{1}{2}\mu(F)$,
then $\mathcal{V}$ is the desired collection and we are done.

Otherwise let $G=F\setminus\cup_{B\in\mathcal{V}}\partial_{2r}B$,
so $\mu(G)>\frac{1}{2}\mu(F)\geq\frac{1}{2}\delta$. By corollary
\ref{cor:measure-disjointification} we may choose a well-separated
sub-collection of balls $\mathcal{U}'\subseteq\mathcal{U}_{k-1}$
with $\mu(\cup\mathcal{U}')>\frac{\delta}{2\chi}\mu(X)$, so by assumption
$\mu(\cup_{B\in\mathcal{U}'}\partial B)>\frac{\varepsilon\delta}{2\chi}\mu(X)$.
Since the centers of $B\in\mathcal{U}'$ are at distance at least
$2r\geq2\maxrad\mathcal{U}'$ from each $S\in\partial\mathcal{V}$,
the collection $\partial\mathcal{V}\cup\partial\mathcal{U}'$ is well-separated,
and we have\[
\mu(\cup_{B\in\mathcal{V}\cup\mathcal{U}'}\partial B)=\mu(\cup_{B\in\mathcal{V}}\partial B)=\mu(\cup_{B\in\mathcal{U}'}\partial B)\geq\frac{\varepsilon\delta}{2\chi}k+\frac{\varepsilon\delta}{2\chi}=\frac{\varepsilon\delta}{2\chi}(k+1)\]
so we complete the recursive step by adding $\mathcal{U}'$ to $\mathcal{V}$.

It only remains to show that this cannot continue for $p$ steps;
and indeed, if it did we would have $\mu(\cup_{B\in\mathcal{V}}\partial B)\geq\mu(X)$,
which is impossible.
\end{proof}

\section{\label{sec:Non-Concentration-of-mass}Coarse dimension and non-Concentration
of mass on boundaries}

For a metric space with a measure, let us say that a ball is $\varepsilon$-thick
if an $\varepsilon$-fraction of its mass is concentrated on its boundary.
In this section we derive a theorem which says, roughly, that given
a finite measure on $\mathbb{R}^{d}$, only a relatively small mass
of points can have the property that they lie at the center of many
$\varepsilon$-thick balls. This result seems to depend on a metric
property that is closely related to topological dimension, which we
call \emph{coarse dimension}. Informally, we wish to express the fact
that the boundary of balls is of a lower dimension than the ambient
space. This is not quite what we need, since we are using thick boundaries
in place of topological boundaries. In general it is not true that
$\partial_{1}B$ has lower dimension than $X$; in $\mathbb{R}^{d}$,
for example, $\partial_{1}B_{r}(x)$ has non-empty interior so it
has full dimension. However from the point of view of balls with radius
$\gg1$, $\partial_{1}B_{r}(x)$ looks more or less like the lower-dimensional
subset $\partial B_{r}(x)$ (and for balls whose radius is $\gg r$,
$\partial_{1}B_{r}(x)$ looks like a point). For this reason we introduce
a parameter $R_{0}$ which specifies how big balls must be in order
to pick up the {}``large scale'' geometry. We make the following
provisional definition, which is neither general nor particularly
elegant, but is convenient for the induction which is to follow.
\begin{defn}
\label{def:coarse-dimension}For metric spaces $X$ and $R_{0}>1$,
the relation $\cdim_{R_{0}}X=k$ (read: $X$ has coarse dimension
$k$ at scales $\geq R_{0})$ is defined by recursion on $k$:
\begin{itemize}
\item $\cdim_{R_{0}}X=-1$ for $X=\emptyset$ and any $R_{0}$,
\item $\cdim_{R_{0}}X=k$ if $\cdim_{R_{0}}X\neq k-1$ and, for every $t\geq1$,
every $r\geq tR_{0}$ and every $x\in X$, the subspace $Y=\partial_{t}B_{r}(x)$
satisfies $\cdim_{tR_{0}}Y=m$ for some $m\leq k-1$.
\end{itemize}
\end{defn}
In showing that $\mathbb{R}^{d}$ has finite coarse dimension we use
a property which is closely related to (and implies) the Besicovitch
property, though the two are apparently not equivalent for general
metric spaces.
\begin{prop}
\label{pro:Rd-lemma}Let $\left\Vert \cdot\right\Vert $ be a norm
on $\mathbb{R}^{d}$. Then there is an $R_{0}>1$ and $k\in\mathbb{N}$
with the following property. Suppose that $r(1)\geq r(2)\geq\ldots\geq r(k)\geq R_{0}$
and $x_{1},x_{2},\ldots,x_{k}\in\mathbb{R}^{d}$ are such that $x_{i}\in\mathbb{R}^{d}\setminus\cup_{j<i}B_{r(j)-1}(x_{j})$.
Then $\cap_{i=1}^{k}\partial_{1}B_{r(i)}(x_{i})=\emptyset$.\end{prop}
\begin{proof}
Let $x\in\mathbb{R}^{d}$ and $r>1$. If $y\in\partial B_{r}(x)$,
then the point $y'=2x-y$ antipodally opposite to $y$ on $\partial B_{r}(x)$
is at distance at least $r$ from any ball $B_{s}(y)$ with $1\leq s\leq r$.
It follows that there is an $\varepsilon>0$ so that, if $z\in\partial B_{r}(x)\cap\partial B_{s}(y)$,
then the angle $\angle xyz$ is greater than $\varepsilon$. By compactness
of $\partial B_{r}(x)$ we can choose $\varepsilon$ uniform in $y$.
By continuity of the map $(u,v,w)\mapsto\angle uvw$, compactness
and the assumption $s\geq1$, we find that for some $\delta>0$, the
same remains true if we perturb $y,z$ by $\delta$.

Since the metric is translation invariant, we have shown the following:
there is a $0<\delta<1$ such that, for any $1\leq s\leq r$ and any
three points $x\in\mathbb{R}^{d}$, $y\in\partial_{\delta/2}B_{r}(x)$
and $z\in\partial_{\delta/2}B_{s}(y)$, the angle $\angle xyz$ is
at least $\varepsilon$. Rescaling and setting $R_{0}=2/\delta$,
we find that if $R_{0}\leq s\leq r$, $y\in\partial_{1}B_{r}(x)$
and $z\in\partial_{1}B_{s}(y)$ then $\angle xyz>\varepsilon$.

Returning to the situation in the formulation of the lemma, if $x\in\cap_{i-1}^{k}\partial_{1}B_{r(i)}(x_{i})$,
then $\angle x_{i}xx_{j}>\varepsilon$ for all $1\leq i<j\leq k$,
and by compactness of the unit sphere this cannot happen for $k$
arbitrarily large. The lemma follows.\end{proof}
\begin{cor}
\label{cor:Rd-coarse-dimension}$\mathbb{R}^{d}$ has finite coarse
dimension with respect to any norm-induced metric.\end{cor}
\begin{proof}
Let $\left\Vert \cdot\right\Vert $ be a fixed norm and let $k',R_{0}$
be the constants as in proposition \ref{pro:Rd-lemma}. Let $k''$
be the size of the maximal $(1-\frac{1}{R_{0}})$-separated set of
points in $B_{2}(0)$. Let $k=k'k''$; we claim that $\cdim_{R_{0}}\mathbb{R}^{d}\leq k$.
Unraveling the definition of coarse dimension, it is apparent that
in order to prove this it suffices to show that if we are given 
\begin{enumerate}
\item \label{enu:t-condition}A sequence $t(1),t(2),\ldots,t(k)\geq1$, 
\item \label{enu:r-condition}a sequence $r(1),r(2),\ldots,r(k)$ such that
$r(i)\geq t(1)\cdot\ldots\cdot t(i)R_{0}$, and
\item \label{enu:x-condition}points $x_{1},x_{2},\ldots,x_{k}\in\mathbb{R}^{d}$
such that $x_{i}\in\partial_{t(j)}B_{r(j)}(x_{j})$ for $j<i$. 
\end{enumerate}
\noindent then $\cap_{i=1}^{k}\partial_{t(i)}B_{r(i)}(x_{i})=\emptyset$. 

First, we claim that we may assume that each of the sequences is of
length $k'$, but that the radii are non-increasing. This will follow
if we show that $r(j)\leq r(1)$ for some $2\leq j\leq k''+1$, because
we can then repeat this with $r(j)$ instead of $r(1)$, and so on
$k'$ times. To show that such a $j$ exists, consider the points
$x_{2},\ldots,x_{k''+1}$ and suppose that $r(j)\geq r(1)$ for $2\leq j\leq k''+1$.
Observe that by (\ref{enu:x-condition}) the $x_{j}$ are all located
within the ball $B_{r(1)+t(1)}(x_{1})\subseteq B_{2r(1)}(x_{1})$,
because by (\ref{enu:r-condition}), $t(1)\leq\frac{r(1)}{R_{0}}\leq r(1)$.
Also, by (\ref{enu:r-condition}) and (\ref{enu:x-condition}), if
$i>j$ then \[
d(x_{i},x_{j})\geq r(j)-t(j)\geq r(j)(1-\frac{1}{R_{0}})\]
so if $r(j)\geq r(1)$ we have $d(x_{i},x_{j})\geq r(1)(1-\frac{1}{R_{0}})$.
Thus $x_{2},\ldots,x_{k''+1}$ is a $r(1)(1-\frac{1}{R_{0}})$-separated
set in the ball $R_{2r(1)}(x_{1})$, and rescaling we obtain a contradiction
to the definition of $k''$.

Now we assume the sequences have length $k'$ and the radii are non-increasing.
Let $t=\max t_{i}$ and replace $\left\Vert \cdot\right\Vert $ with
$\left\Vert \cdot\right\Vert ^{*}=\frac{1}{t}\left\Vert \cdot\right\Vert $.
After this rescaling, we wish to show that $\cap_{i=1}^{k'}\partial_{t_{i}/t}^{*}B_{r(i)/t}^{*}(x_{i})=\emptyset$,
where the $*$'s indicate operations with respect to $\left\Vert \cdot\right\Vert ^{*}$.
For this it is enough to show that $\cap_{i=1}^{k'}\partial_{1}^{*}B_{r(i)/t}^{*}(x_{i})=\emptyset$,
and this will follow once we verify the hypothesis of the previous
proposition for the norm $\left\Vert \cdot\right\Vert ^{*}$; and
indeed, clearly $r(i)/t$ is still decreasing; $x_{i}\in\partial_{t(j)}B_{r(j)}(x_{j})$
implies $x_{i}\in\partial_{1}^{*}B_{r(j)/t}^{*}(x_{j})$; and the
inequalities $r(i)\geq r(k)\geq t(1)\cdot\ldots\cdot t(k')R_{0}$
and $t_{i}\geq1$ imply $r(i)/t\geq R_{0}$, as required there.
\end{proof}
We can now state and prove the main result of this section, which
is the main tool in the proof of theorem \ref{thm:chacon-ornstein}.
Although it is $\mathbb{R}^{d}$ that we have in mind, the formulation
is for general metric spaces in order to facilitate the inductive
proof.
\begin{thm}
\label{thm:main-lemma}Fix $k,\chi\in\mathbb{N}$ and $0<\varepsilon,\delta<1$
and set $q=(\frac{200\chi^{2}}{\varepsilon^{2}\delta^{3}})^{k}\cdot1000^{k^{2}}$.
Suppose that
\begin{enumerate}
\item \label{enu:metric-space-parameters}$X$ is a metric space with $\chi(X)\leq\chi$
and $\cdim_{R_{0}}X=k$ for some $R_{0}>2$, 
\item $\mu$ is a finite measure on $X$,
\item $F\subseteq X$ is finite,
\item \label{enu:stack-properties}$\mathcal{U}_{1},\mathcal{U}_{2}\ldots,\mathcal{U}_{q}$
is a stack over $F$ with 

\begin{enumerate}
\item \label{enu:stack-growth}$\minrad\mathcal{U}_{i}\geq(\maxrad\mathcal{U}_{i-1})^{2}$, 
\item \label{enu:stack-minrad}$\minrad\mathcal{U}_{1}\geq\max\{2,R_{0}\}$,
\end{enumerate}
\item \label{enu:large-boundary}$\mu(\partial_{1}B)\geq\varepsilon\mu(B)$
for each $B\in\cup_{i}\mathcal{U}_{i}$.
\end{enumerate}
Then $\mu(F)\leq\delta\mu(X)$.\end{thm}
\begin{rem*}
No attempt has been made to optimize the conditions. A slower rate
of growth in (\ref{enu:stack-growth}) would probably suffice.\end{rem*}
\begin{proof}
Define integers $Q(k,\chi,\varepsilon,\delta)$ recursively by\begin{eqnarray*}
Q(0,\chi,\varepsilon,\delta) & = & 1\\
Q(k,\chi,\varepsilon,\delta) & = & \left\lceil \frac{2\chi}{\varepsilon\delta}\right\rceil \cdot(1+\left\lceil \frac{64\chi}{\varepsilon\delta^{2}}\right\rceil )\cdot(1+Q(k-1,\chi,\frac{\varepsilon}{2},\frac{\delta}{8})\end{eqnarray*}
One may verify that $q\geq Q(k,\chi,\varepsilon,\delta)$, so it suffices
to prove the claim for $q=Q(k,\chi,\varepsilon,\delta)$; this we
do by induction on $k$.

For $k=0$ the claim is trivial, since then $\partial B=\emptyset$
for any ball $B$. We can therefore have $\mu(\partial B)\geq\varepsilon\mu(B)$
only when $\mu(B)=0$, implying that each point in $F$ has mass $0$,
so $\mu(F)=0$. 

Assume that the claim holds for $k-1$. We suppose that $X,R_{0},\chi,\mu,F,\mathcal{U}_{1},\ldots,\mathcal{U}_{q}$
satisfy the hypotheses of the theorem but $\mu(F)>\delta\mu(X)$,
and proceed to derive a contradiction.

\noindent \textbf{Preliminary disjointification}: We first pass to
a sub-sequence of the given carpets and extract a disjoint family
of balls from them. Let \[
N=q/\left\lceil \frac{2\chi}{\varepsilon\delta}\right\rceil \]
Since $q/N=\left\lceil \frac{2\chi}{\varepsilon\delta}\right\rceil $,
we may apply corollary \ref{cor:measure-disjointification} to the
stack $\{\mathcal{U}_{iN}\}_{1\leq i\leq q/N}$ obtained by choosing
each $N$-th element of the original stack. We get an $n_{0}\geq0$
and a collection \[
\mathcal{V}\subseteq\bigcup_{i\geq N(n_{0}+1)}\mathcal{U}_{i}\]
such that $\partial\mathcal{V}$ is well-separated, and such that,
setting\[
r=\maxrad\mathcal{U}_{n_{0}}\]
we have \[
\mu(F\cap\bigcup_{B\in\mathcal{V}}\partial_{2r}B)\geq\frac{1}{2}\mu(F)>\frac{\delta}{2}\mu(X)\]
We denote the union on the left hand side by\[
Y=\bigcup_{B\in\mathcal{V}}\partial_{2r}B\]
From now on we can forget about the carpets $\mathcal{U}_{i}$ for
$i<n_{0}N$ and $i>(n_{0}+1)N$; we work only with $\mathcal{V}$
and $\mathcal{U}_{i}$, $n_{0}M\leq i\leq(n_{0}+1)N$.

\noindent \textbf{Outline of the argument}. Roughly, our argument
proceeds as follows. The set $Y$ is made up of a union of thick spheres,
and each is of lower coarse dimension than $X$. Let $S$ be one of
these spheres, and suppose that some nontrivial fraction of its mass
comes from $F$. Consider the stack obtained by fixing a large $p$
(but still much smaller than $N$) and selecting from the stack $\mathcal{U}_{n_{0}N},\mathcal{U}_{n_{0}N+1}\ldots,\mathcal{U}_{n_{0}N+p}$
those balls centered in $F\cap S$. The induction hypothesis can be
applied to show that for a nontrivial fraction of points $x\in F\cap S$
there is a ball in this stack whose $1$-boundary \emph{with respect
to $S$ }contains only an $\varepsilon/2$-fraction of the balls mass.
However, with respect to $X$ these $1$-boundaries contain an $\varepsilon$-fraction
of the mass. Therefore, the difference -- an $\varepsilon/2$ of the
balls' mass -- lies outside $S$. Passing to a disjoint sub-collection
of these balls centered in $S$, we obtain a set of mass equal to
some nontrivial fraction of $F\cap S$, located outside of $S$ but
nearby it. Now, since the mass of $F\cap Y$ is large, the situation
described can be repeated for spheres $S$ comprising a non-negligible
fraction of $Y$, and the masses obtained outside each sphere will
be disjoint from each other. We conclude that, in the near vicinity
of $Y$ but disjoint from $Y$ there is a set with mass a small but
constant fraction of $\mu(X)$. Next, we repeat this argument, replacing
$Y$ with a small neighborhood of $Y$, and using the next $p$ carpets
$\mathcal{U}_{n_{0}N+p+1},\ldots,\mathcal{U}_{n_{0}N+2p}$, and get
another mass increment. After doing this sufficiently many times we
will have accumulated more mass than there is in $X$ altogether,
a contradiction. 

\noindent \textbf{Partitioning into further sub-stacks}. Let us denote\[
r_{i}^{+}=\maxrad\mathcal{U}_{i}\qquad,\qquad r_{i}^{-}=\minrad\mathcal{U}_{i}\]
and set\[
p=Q(k-1,\chi,\frac{\varepsilon}{2},\frac{\delta}{8})\]
We partition the carpets $\{\mathcal{U}_{i}\}_{n_{0}N+1\leq i\leq(n_{0}+1)N-1}$,
into sub-stacks of height $p+1$. More precisely, let $M=N/(p+1)$,
and for $0\leq j\leq M-1$ define \[
m(j)=n_{0}N+1+(p+1)\cdot j\]
so for each such $j$ we get the stack $\{\mathcal{U}_{m(j)+i}\}_{1\leq i\leq p}$.
The first thing to note is that all the balls in these stacks are
from carpets $\mathcal{U}_{i}$ below $\mathcal{U}_{(n_{0}+1)N}$,
which is the level where $\mathcal{V}$ begins. Consequently, the
radii of all these balls is much smaller than the radii of balls in
$\mathcal{V}$; indeed, the largest possible radius in our sub-stacks
is \[
r_{m(M-1)+p}^{+}=r_{(n_{0}+1)N-1}^{+}<\frac{1}{r_{(n_{0}+1)N-1}^{+}}r_{(n_{0}+1)N}^{+}<\frac{1}{2}\minrad\mathcal{V}\]
by (\ref{enu:stack-growth}) and (\ref{enu:stack-minrad}). 

\noindent \textbf{Thickening the set $Y$}: For $0\leq j\leq M-1$
it will be convenient to denote\[
\Delta_{j}B=\partial_{r_{m(j)}^{+}}B\]
and to thicken the set $Y$ by thickening each sphere in $Y$, obtaining
\[
Y_{j}=\bigcup_{B\in\mathcal{V}}\Delta_{j}B\]
Let us note several properties of the $Y_{j}$. First, clearly $Y_{1}\subseteq Y_{2}\subseteq\ldots\subseteq Y_{M}$,
and \[
Y\subseteq Y_{j}\]
implying\[
\mu(F\cap Y_{j})\geq\mu(F\cap Y)\geq\frac{\delta}{2}\mu(X)\]
To see this it is enough to show that $Y\subseteq Y_{0}$, and indeed
by (\ref{enu:stack-growth}) and (\ref{enu:stack-minrad}), \[
r_{m(0)}^{+}=r_{n_{0}N+1}^{+}>(r)^{2}\geq2r\]
so, since $Y_{0},Y$ are obtained, respectively, as the $r_{m(0)}^{+}$-thickening
and $r$-thickening of the same spheres, the claim follows. 

Second, each $Y_{j}$ is the disjoint union of the thick spheres $\Delta_{j}B$,
$B\in\mathcal{V}$. This follows from the inequality $r_{m}^{+}(j)<\frac{1}{2}\minrad\mathcal{V}$,
noted above,  and the fact that $\mathcal{V}$ is well-spaced.

Third, let $0\leq j\leq M-2$, and $x\in\partial_{r_{m(j)}^{+}}B$
for some $B\in\mathcal{V}$. Suppose that $B'\in\mathcal{U}_{m(j)+i}$
for some $1\leq i\leq p$ is centered at $x$; then $\partial_{1}B'\subseteq\Delta_{j+1}B$.
To see this, suppose that $y\in B$ and $z\in\partial_{1}B'$. Then\[
d(y,z)\leq d(y,x)+d(x,z)\leq r_{m(j)}^{+}+r_{m(j)+i}^{+}+1\leq2r_{m(j)+p}^{+}<r_{m(j+1)}^{+}\]
which proves the claim.

\noindent \textbf{Spheres in $Y_{j}$ containing a large proportion
of $F$}. For each $0\leq j\leq M-1$, define \[
\mathcal{\mathcal{W}}_{j}=\{B\in\mathcal{V}\,:\,\mu(F\cap\Delta_{j}B)>\frac{\delta}{4}\mu(\Delta_{j}B)\}\]
A Markov-type argument now shows that the spheres in $\mathcal{W}_{j}$
contain a large fraction of $X$: \begin{eqnarray*}
\frac{\delta}{2}\mu(X) & < & \mu(F\cap Y_{j})\\
 & = & \mu(F\cap\bigcup_{B\in\mathcal{V}}\Delta_{j}B)\\
 & = & \mu(F\bigcup_{B\in\mathcal{W}_{j}}\Delta_{j}B)+\mu(F\cap\bigcup_{B\in\mathcal{V}\setminus\mathcal{W}_{j}}\Delta_{j}B)\\
 & \leq & \mu(F\cap\bigcup_{B\in\mathcal{W}_{j}}\Delta_{j}B)+\frac{\delta}{4}\mu(\bigcup_{B\in\mathcal{V}\setminus\mathcal{W}_{j}}\Delta_{j}B)\\
 & \leq & \mu(F\cap\bigcup_{B\in\mathcal{W}_{j}}\Delta_{j}B)+\frac{\delta}{4}\mu(X)\end{eqnarray*}
and, rearranging, we get\[
\mu(F\cap\bigcup_{B\in\mathcal{W}_{j}}\Delta_{j}B)>\frac{\delta}{4}\mu(X)\]

\noindent \textbf{Applying the induction hypothesis to fat spheres}.
Fix $0\leq j\leq M-1$ and let $S=\Delta_{j}B$ for some $B\in\mathcal{W}_{j}$.
Put $\mu_{S}=\mu|_{S}$ i.e. $\mu_{S}(A)=\mu(A\cap S)$, and $F_{S}=F\cap S$.

Consider the stack $\{\mathcal{U}'_{t}\}_{1\leq t\leq p}$ over $F_{S}$
obtained by selecting from $\{\mathcal{U}_{m(j)+t}\}_{1\leq t\leq p}$
those balls with centers in $F_{S}$. This is a stack in $X$, but
from it we get a stack in $S$ by intersecting each ball with $S$. 

We claim that $S$, $\mu_{S}$, $F_{S}$ and this stack satisfy conditions
(\ref{enu:metric-space-parameters}) to (\ref{enu:stack-minrad})
of the theorem, with $k-1$ in place of $k$ and $r_{m(j)}^{+}R_{0}$
in place of $R_{0}$. Indeed, $S=\partial_{r_{m(j)}^{+}}B$, so it
has coarse dimension $\leq k-1$ at scales $\geq r_{m(j)}^{+}R_{0}$,
and $\chi(S)\leq\chi(X)$ because $S\subseteq X$. The relative growth
of radii in $\mathcal{U}'_{i}$ is inherited from $\mathcal{U}_{i}$.
Finally, $\mathcal{U}'_{1}\subseteq\mathcal{U}_{m(j)+1}$, so by the
growth assumption for the original stack, \[
\minrad\mathcal{U}'_{1}\geq r_{m(j)+1}^{-}\geq(r_{m(j)}^{+})^{2}\geq r_{m(j)}^{+}R_{0}\]
and clearly also $\minrad\mathcal{U}'_{1}\geq2$, which verifies (\ref{enu:stack-minrad}).

Let \[
F'_{S}=\{x\in F_{S}\,:\,\mu_{S}(\partial_{1}B')\geq\frac{\varepsilon}{2}\mu_{S}(B')\mbox{ for every }B'\in\mathcal{U}'_{t}\,,\,1\leq t\leq p\}\]
Applying the induction hypothesis to the stack obtained by restricting
each $\mathcal{U}'_{t}$ to balls with center in $F'_{S}$, and recalling
the definition of $p$, we find that \[
\mu_{S}(F'_{S})\leq\frac{\delta}{8}\mu_{S}(S)\]
Since $S=\Delta_{j}B$ for some $B\in\mathcal{W}_{j}$ we know that
$\mu_{S}(F_{S})>\frac{\delta}{4}\mu_{S}(S)$; so \[
\mu_{S}(F_{S}\setminus F'_{S})>\frac{\delta}{8}\mu_{S}(S)=\frac{\delta}{8}\mu(S)\]

\noindent \textbf{Estimating the mass outside of a fat sphere}. For
each $x\in F_{S}\setminus F'_{S}$ there is some $1\leq t\leq p$
and $B'\in\mathcal{U}'_{t}$, centered at $x$, with \[
\mu(\partial_{1}B'\cap S)=\mu_{S}(\partial_{1}B')\leq\frac{\varepsilon}{2}\mu_{S}(B'\cap S)\]
But by (5) we have \[
\mu(\partial_{1}B')\geq\varepsilon\mu(B')\geq\varepsilon\mu_{S}(B')\]
therefore, \[
\mu(\partial_{1}B'\setminus S)\geq\frac{\varepsilon}{2}\mu(B'\cap S)\]

\noindent \textbf{Estimating the mass between $Y_{j}$ and $Y_{j+1}$}.
Applying corollary \ref{cor:measure-disjointification} to each of
the balls above as $x$ runs over $F_{S}\setminus F'_{S}$, we choose
a disjoint collection $\mathcal{C}$ of balls centered in $F_{S}\setminus F'_{S}$
satisfying the last inequality, and which cover a $\frac{1}{\chi}$-fraction
of $F_{S}\setminus F'_{S}$ and so has mass $>\frac{\delta}{8\chi}\mu(S)$.
The corresponding union of $1$-spheres, since each contains an $\varepsilon$-fraction
of the mass of the solid ball, has mass $>\frac{\varepsilon\delta}{8\chi}\mu(S)$;
and since at least half this mass lies outside of $S$, we get\[
\mu((\bigcup_{B'\in\mathcal{C}}\partial_{1}B')\setminus S)\geq\frac{\varepsilon\delta}{16\chi}\mu(S)\]

\noindent The set on the left hand side is in the complement of $S=\Delta_{j}B$,
but certainly lies inside $\Delta_{j+1}B$, and these sets are disjoint
for distinct $B\in\mathcal{V}$. So the contribution of mass near
each $S$ is disjoint from the contributions of other $S$'s, so\begin{eqnarray*}
\mu(Y_{j+1}\setminus Y_{j}) & \geq & \sum_{B\in\mathcal{W}_{j}}\mu(\Delta_{j+1}B\setminus\Delta_{j}B)\\
 & \geq & \sum_{B\in\mathcal{W}_{j}}\frac{\varepsilon\delta}{16\chi}\mu(\Delta_{j}B)\\
 & = & \frac{\varepsilon\delta}{16\chi}\mu(\bigcup_{B\in\mathcal{W}_{j}}\Delta_{j}B)\\
 & \geq & \frac{\varepsilon\delta}{16\chi}\cdot\frac{\delta}{4}\mu(X)\end{eqnarray*}
because $\mu(\cup_{B\in\mathcal{W}_{j}}\Delta_{j}B)\geq\mu(Y_{j})>\frac{\delta}{4}\mu(X)$.

\noindent \textbf{The punchline}. The number of $Y_{j}$'s has been
arranged to be\[
M=\frac{N}{p+1}=\frac{q}{\left\lceil 2\chi/\varepsilon\delta\right\rceil (p+1)}\geq\frac{64\chi}{\varepsilon\delta^{2}}+1\]
so the relation $Y_{j}\subseteq Y_{j+1}$ and $\mu(Y_{j+1}\setminus Y_{j})>\frac{\varepsilon\delta}{16\chi}\cdot\frac{\delta}{4}\mu(X)$
imply \[
\mu(Y_{M})\geq M\cdot\frac{\varepsilon\delta}{16\chi}\cdot\frac{\delta}{4}\mu(X)>\mu(X)\]
which is the desired contradiction.
\end{proof}
It is not hard to see that a similar result holds if $\mu$ is a Borel
measure, $F$ is a Borel set and the carpets are measurable (i.e.
the function $r_{n}:F\rightarrow\mathbb{R}^{+}$ describing the radii
of the balls in the $n$-th carpet is measurable). One way to see
this is to discretize the data. For a fine partition $\mathcal{P}=\{P_{i}\}$
of $X$, choose a representative $x_{i}\in P_{i}$ in each atom, set
$r'_{n}(x)=\int_{P_{i}}r_{n}$, and replace $\mu$ with the atomic
measure supported on the $x_{i}$ with $\mu'(\{x_{i}\})=\mu(P_{i})$.
Applying the discrete lemma above to $\mu'$ and the new stack, with
suitably modified parameters, we can deduce the result for the original
measure.

\section{\label{sec:Proof-of-the-ratio-theorem}Proof of the ratio theorem
for $\mathbb{Z}^{d}$}

Given what we have proved so far, theorems \ref{thm:chacon-ornstein}
and \ref{thm:main} now follow by fairly standard arguments.

\subsection*{Proof of theorem \ref{thm:chacon-ornstein}: }

This is a standard application of the transference together with theorem
\ref{thm:main-lemma}. Let $\mathbb{Z}^{d}$ act on a $\sigma$-finite
measure space $(\Omega,\mathcal{B},\mu)$ by non-singular transformations.
By passing to an equivalent measure, we may assume that $\mu(\Omega)=1$.
Let $T$ act by translation on $L^{\infty}$, and let $\widehat{T}$
be the dual action of $T$ on $L^{1}\subseteq(L^{\infty})^{*}$, which
is a linear, order-preserving isometry defined by the condition $\int\widehat{T}^{u}f\cdot gd\mu=\int f\cdot T^{u}gd\mu$
for $f\in L^{\infty}$ and $g\in L^{1}$, and explicitly by $\widehat{T}^{u}f(\omega)=f(T^{-u}\omega)\cdot\frac{dT^{u}\mu}{d\mu}$.

Fix a norm $\left\Vert \cdot\right\Vert $ on $\mathbb{R}^{d}$, and
suppose $\cdim_{R_{0}}\mathbb{R}^{d}=k$ and $\chi(\mathbb{R}^{d})=\chi$
for appropriate parameters $R_{0},k,\chi$. 

Let $1\leq f\in L^{\infty}\subseteq L^{1}$. We are out to prove that
\[
s_{n}(\omega)=\frac{\sum_{u\in\partial_{1}B_{n}}\widehat{T}^{u}f(\omega)}{\sum_{u\in B_{n}}\widehat{T}^{u}f(\omega)}\rightarrow0\]
for a.e. $\omega$ (from this the case of thick boundaries $\partial_{t}$
follows by rescaling the norm). Set \[
A_{\varepsilon}=\{\omega\in\Omega\,:\,\limsup s_{n}(\omega)>\varepsilon\}\]
and suppose that $\mu(A_{\varepsilon})>0$ for some $\varepsilon$.
We construct a sequence \[
R_{0}=r_{0}^{-}=r_{0}^{+}\leq r_{1}^{-}\leq r_{1}^{+}\leq r_{2}^{-}\leq r_{2}^{+}\leq\ldots\]
satisfying $r_{i}^{-}\geq(r_{i-1}^{+})^{2}$ and $r_{1}^{-}\geq\max\{2,R_{0}\}$,
and a set of points $A\subseteq A_{\varepsilon}$, so that for every
$\omega\in\Omega$ and $i\geq1$ there is an $n_{i}=n_{i}(\omega)\in(r_{i}^{-},r_{i}^{+})$
with $s(n_{i},\omega)>\varepsilon$, and $\mu(A)>\frac{1}{2}\mu(A_{\varepsilon})$.
We do this by recursion, so that going into the $i$-th stage we have
defined $r_{j}^{\pm}$ for $j<i$ and sets $C_{0}\subseteq C_{1}\subseteq\ldots\subseteq C_{i-1}\subseteq A_{\varepsilon}$
satisfying the above and $\mu(C_{j})\geq(\frac{1}{2}+\frac{1}{j+1})\mu(A_{\varepsilon})$.
In order to define $r_{i}^{\pm}$ and $C_{i}$, first set $r_{i}^{-}=(2\vee r_{i-1}^{+})^{2}$.
Now, since $C_{i-1}\subseteq A_{\varepsilon}$, for every $\omega\in C_{i-1}$
there is an $n=n(\omega)\geq r_{i}^{-}$ with $s(n,\omega)>\varepsilon$;
so we can choose $r_{i}^{+}$ so that $n(\omega)\leq r_{i}^{+}$ on
a subset of $C_{i-1}$ of measure $>(\frac{1}{2}+\frac{1}{i+1})\mu(A_{\varepsilon})$.
This set will be $C_{i}$, and $A=\cap_{j=1}^{\infty}C_{j}$.

We are now ready to apply the transference principle. Fix $\delta$
and $n>r_{q}^{+}$, where $q=q(k,\chi,\varepsilon,\delta)$ is as
in theorem \ref{thm:main-lemma}. Then\[
\mu(A)=\int1_{A}d\mu=\frac{1}{|B_{n}|}\sum_{u\in B_{n}}\int\widehat{T}^{u}1_{A}d\mu=\frac{1}{|B_{n}|}\int\sum_{u\in B_{n}}\widehat{T}^{u}1_{A}d\mu\]
Next, we bound the sum $\sum_{u\in B_{n}}\widehat{T}^{u}1_{A}$. Fix
$\omega\in\Omega$ and consider the measure $\nu=\nu_{\omega,n}$
on $B_{2n}$ by $\nu(\{u\})=\widehat{T}^{u}f(\omega)$. Let \[
U=U_{\omega,n}=\{u\in B_{n}\,:\, T^{-u}\omega\in A\}\]
By the definition of $A$ there is a stack of height $q$ over $U$
satisfying the hypothesis of theorem \ref{thm:main-lemma}, and all
the balls in the stack are of radius $\leq r_{q}^{+}<n$, implying
that they are contained in $B_{2n}$. Thus, by theorem \ref{thm:main-lemma},
\begin{eqnarray*}
\nu(U) & \leq & \delta\nu(B_{2n})\\
 & = & \delta\sum_{u\in B_{2n}}\widehat{T}^{u}f(\omega)\\
 & \leq & \delta\left\Vert f\right\Vert _{\infty}\sum_{u\in B_{2n}}\widehat{T}^{u}1\end{eqnarray*}
We have arranged things so that \[
\sum_{u\in B_{n}}\widehat{T}^{u}1_{A}(\omega)=\nu_{\omega,n}(U_{\omega,n})\]
therefore \begin{eqnarray*}
\sum_{u\in B_{n}}\widehat{T}^{u}1_{A}(\omega) & \leq & \nu_{\omega,n}(U_{\omega,n})\\
 & \leq & \delta\left\Vert f\right\Vert _{\infty}\sum_{u\in B_{2n}}\widehat{T}^{u}1\end{eqnarray*}
dividing by $|B_{n}|$ and integrating we get\begin{eqnarray*}
\mu(A) & \leq & \frac{1}{|B_{n}|}\int\delta\left\Vert f\right\Vert _{\infty}\sum_{u\in B_{2n}}\widehat{T}^{u}1d\mu\\
 & \leq & \frac{\delta\left\Vert f\right\Vert _{\infty}}{|B_{n}|}|B_{2n}|\int1d\mu\\
 & \leq & 2^{d}\left\Vert f\right\Vert _{\infty}\delta\end{eqnarray*}
because $\frac{|B_{n}|}{|B_{2n}|}\leq2^{d}$. The right hand side
can be made arbitrarily small, so $\mu(A)=0$; hence also $\mu(A_{\varepsilon})=0$.

Finally, $s_{n}(\omega)\rightarrow0$ if and only if $\omega\notin\cup_{m=1}^{\infty}A_{1/m}$,
and the set on the right is seen to have measure $0$. This completes
the proof of theorem \ref{thm:chacon-ornstein}.\hfill\qedsymbol

\subsection*{Proof of theorem \ref{thm:main}:}

The proof is standard. We first prove the case $g\equiv1$. Consider
the space \[
\mathcal{F}=\mbox{span}\{1,f-\widehat{T}^{v}f\,:\, v\in\mathbb{Z}^{d}\mbox{ and }f\in L^{\infty}\}\]
One shows that $\mathcal{F}$ is dense in $L^{1}$; the proof follows
the same lines as Riesz's proof of the mean ergodic theorem, using
the duality relation $(L^{1})^{*}=L^{\infty}$ instead of self-duality
of $L^{2}$. See \cite{F07,A97}. 

Next, one shows that the ratios $R_{n}(f,1)$ converge for every member
of $\mathcal{F}$. Indeed, note that $R_{n}(1,1)\equiv1$; whereas
if $f\in L^{\infty}$ then the ratios $R_{n}(f-T^{v}f,1)$ satisfy
\begin{eqnarray*}
|\frac{\sum_{u\in B_{n}}\widehat{T}^{u}(f-\widehat{T}^{v}f)}{\sum_{u\in B_{n}}\widehat{T}^{u}1}| & \leq & \frac{\sum_{u\in\partial_{\left\Vert v\right\Vert }B_{n}}\widehat{T}^{u}|f|}{\sum_{u\in B_{n}}\widehat{T}^{u}1}\\
 & = & \frac{\sum_{u\in\partial_{\left\Vert v\right\Vert }B_{n}}\widehat{T}^{u}|f|}{\sum_{u\in B_{n}}\widehat{T}^{u}|f|}\cdot\frac{\sum_{u\in B_{n}}\widehat{T}^{u}|f|}{\sum_{u\in B_{n}}\widehat{T}^{u}1}\\
 & \leq & \frac{\sum_{u\in\partial_{\left\Vert v\right\Vert }B_{n}}\widehat{T}^{u}|f|}{\sum_{u\in B_{n}}\widehat{T}^{u}|f|}\cdot\left\Vert f\right\Vert _{\infty}\end{eqnarray*}
and the right hand side converges to $0$ a.e. by theorem \ref{thm:chacon-ornstein}.
From this it follows that $R_{n}(f,1)\rightarrow\int f$ for any $f\in\mathcal{F}$. 

The case $g\equiv1$ is concluded by applying the maximal inequality
to get convergence on the closure of $\mathcal{F}$, which is all
of $L^{1}$. This standard argument can be found in \cite{A97}. It
is also easy to check that the correct limit is obtained.

Finally, the case of general $g\in L^{1}$ is deduced from the equality
$R_{n}(f,g)=R_{n}(f,1)/R_{n}(g,1)$. \hfill\qedsymbol

\bibliographystyle{plalpha}
\bibliography{bib}

\end{document}